\documentclass[a4paper,oneside,reqno]{amsart}

\calclayout

\usepackage{etoolbox}

\newtoggle{default}\toggletrue{default}
\newtoggle{birk}\togglefalse{birk}
\newtoggle{birk_t2}\togglefalse{birk_t2}
\newtoggle{siam}\togglefalse{siam}
\newtoggle{springer}\togglefalse{springer}

\newtoggle{endwithsymbol}\togglefalse{endwithsymbol}

\toggletrue{default}
\toggletrue{endwithsymbol}

\usepackage[utf8]{inputenc}
\usepackage[T1]{fontenc}
\usepackage{lmodern}
\usepackage[english]{babel}
\usepackage{microtype} 

\usepackage{amsmath,amssymb,amsthm}
\usepackage{thmtools} 

\usepackage[hypertexnames=false]{hyperref} 
\hypersetup{hidelinks}
\usepackage[nameinlink,sort]{cleveref} 
\usepackage{orcidlink}

\usepackage{mathtools} 
\usepackage{upgreek}
\usepackage{accents} 
\usepackage{stmaryrd} 

\usepackage{tikz} 
\usetikzlibrary{cd}
\usetikzlibrary{calc}

\usepackage{dsfont} 
\usepackage{pifont} 

\usepackage{ifthen}
\usepackage{enumitem}
\setlist[itemize]{labelindent=\parindent,leftmargin=*,itemsep=3pt,topsep=5pt}
\setlist[enumerate]{label=\textup{(\roman{*})},itemsep=3pt,labelindent=0pt,leftmargin=*}

\usepackage{nicematrix}

\usepackage{graphicx} 
\usepackage{subcaption}

\usepackage{booktabs} 
\usepackage{esint} 

\usepackage{bbm}

\usepackage{color}

\usepackage{pgfpages} 
\usepackage{pgfplots} 
\pgfplotsset{compat=1.13} 

\definecolor{NathanaelColor}{rgb}{0.2,0.0,0.6}

\renewcommand{\leq}{\leqslant}
\renewcommand{\geq}{\geqslant}

\newcommand{\R}{\mathbb{R}} 
\newcommand{\N}{\mathbb{N}} 
\newcommand{\Q}{\mathbb{Q}} 
\newcommand{\C}{\mathbb{C}} 
\newcommand{\CC}{\mathbb{C}}

\newcommand{\Cc}[1][\infty]{\mathrm{C}_{\mathrm{c}}\ifthenelse{\equal{#1}{}}{}{^{#1}}}
\newcommand{\Lp}[2][]{\mathrm{L}_{#2\ifthenelse{\equal{#1}{}}{}{,#1}}} 
\newcommand{\Lb}{\mathcal{L}_{\mathrm{b}}} 
\newcommand{\sobH}{\mathrm{H}}
\newcommand{\cH}{\accentset{\circ}{\sobH}}

\newcommand{\iu}{\mathrm{i}} 

\DeclareMathOperator{\ran}{ran}

\DeclareMathOperator{\spn}{span}

\DeclareMathOperator{\sym}{sym}
\DeclareMathOperator{\dev}{dev}
\DeclareMathOperator{\dom}{dom}

\DeclareMathOperator{\grad}{grad}
\DeclareMathOperator{\dive}{div}
\DeclareMathOperator{\curl}{curl}

\newcommand{\cgrad}{\mathop{\accentset{\circ}{\grad}}}
\newcommand{\ccurl}{\mathop{\accentset{\circ}{\curl}}}

\newcommand{\cgradgrad}{\mathop{\accentset{\circ}{\operatorname{He}}}}

\DeclarePairedDelimiterX{\dset}[2]{\{}{\}}{#1\,\delimsize\vert\,\mathopen{} #2}
\DeclarePairedDelimiterX{\scprod}[2]{\langle}{\rangle}{#1,#2}

\renewcommand{\Re}{\operatorname{Re}}

\newcommand{\Htopo}{\mathrm{H}}
\newcommand{\nlHtopo}{\mathrm{nlH}}

\theoremstyle{plain}
\newtheorem{theorem}{Theorem}[section]
\newtheorem*{mtheorem}{Free Lunch Theorem}
\newtheorem{lemma}[theorem]{Lemma}

\newtheorem{corollary}[theorem]{Corollary}

\iftoggle{endwithsymbol}{%
    \declaretheorem[style=definition,sibling=theorem,qed=\ding{169}]{definition}
    \declaretheorem[style=definition,sibling=theorem,qed=\ding{169}]{example}
    \declaretheorem[style=definition,sibling=theorem,qed=\ding{169}]{problem}
    \declaretheorem[style=definition,sibling=theorem,qed=\ding{169}]{assumption}
    \declaretheorem[style=definition,numbered=no,qed=\ding{169}]{claim}

    \declaretheorem[style=remark,sibling=theorem,qed=\ding{169}]{remark}
  }%
  {%
    \theoremstyle{definition}
    
    \newtheorem{example}[theorem]{Example}

    \theoremstyle{remark}
    \newtheorem{remark}[theorem]{Remark}
  }%

\begin{document}

\title{The Free Lunch Theorem of Homogenisation} 

\author[A.~Buchinger]{Andreas Buchinger\,\orcidlink{0009-0004-4203-5874}}
\address[A.B.]{Technische Universität Hamburg\\
  Institut für Mathematik\\
  Am Schwarzenberg-Campus 3\\
  D-21073 Hamburg\\
  Germany}
  \email{andreas.buchinger@tuhh.de}

\author[M.~Porfido]{Marianna Porfido\,\orcidlink{0009-0000-0695-9894}}
\address[M.P.]{TU Bergakademie Freiberg \\
  Institute of Applied Analysis \\
  Akademiestrasse 6 \\
  D-09596 Freiberg \\
  Germany}
\email{marianna.porfido@math.tu-freiberg.de}

\author[M.~Waurick]{Marcus Waurick\,\orcidlink{0000-0003-4498-3574}}
\address[M.W.]{TU Bergakademie Freiberg \\
  Institute of Applied Analysis \\
  Akademiestrasse 6 \\
  D-09596 Freiberg \\
  Germany}
\email{marcus.waurick@math.tu-freiberg.de}

\date{\today}
\dedicatory{}



\ifboolexpr{togl{birk_t2} or togl{birk}}{%
}{%
}%

\thanks{The second author is member of the
Gruppo Nazionale per l’Analisi Matematica, la Probabilit\`a e le loro Applicazioni (GNAMPA) of the Istituto Nazionale di Alta Matematica (INdAM)}


\begin{abstract} We show that H-convergence for multiplication type operators as envisioned by Murat and Tartar in the 1970's always implies nonlocal H-convergence as introduced in 2018 in Calc.~Var.~PDE 57(6):159. In contrast to earlier findings, the results presented here work for arbitrary space dimensions, are not bound to a certain geometry of the underlying domain, and do not explicitly require an underlying Hilbert complex for the application of any particular version of the div-curl lemma. We extend classical theory and the main results to more general differential operators with different boundary conditions and orders. Furthermore, the present results confirm homogenisation formulas used in the literature of which we failed to find an explicit proof. As a consequence, H-convergence for multiplication operators in divergence form problems will always imply H-type convergence for a different variational problem for free.
  \end{abstract}

\ifboolexpr{togl{default} or togl{birk_t2} or togl{birk}}{\maketitle}{}%

\section{Introduction}\label{sec:intro}

In the 1840s in British and American taverns a tactic to incentivise customers to frequent these bars was to offer lunch for free in order to have them indirectly pay for it by buying expensive drinks. This rather obvious trick was later abbreviated by `there is no free lunch’ and, thus, becoming a widely used proverb in the English language. 

In a nutshell, in contrast to everyday experience and 19th century taverns, the main contribution of the present article is the following meta-theorem, which we coin the `Free Lunch Theorem' for ease of loose reference, where the presence of convergence in a homogenisation sense always yields an associated convergence for free.
\begin{mtheorem}
Given a sequence of multiplication operators $(a_n)_{n\in \N}$ converging to a limit operator $a$ in some homogenisation sense. Then there exists a sequence of associated problems such that $(a_n^{-1})_{n\in \N}$ converges to $a^{-1}$ in a similar homogenisation sense. 
\end{mtheorem}
It is the main aim of this introduction to make precise the meaning and impact of the Free Lunch Theorem. We start with the classical situation for divergence form problems: In order to analyse non-periodic homogenisation problems for non-selfadjoint multiplication operators, Tartar and Murat coined the notion of \emph{$\Htopo$-convergence} (see, e.g., \cite{MuTa97,Ta09}). 

Note that all inner products considered in this paper are linear in the second entry, conjugate linear in the first, and $d\in\N\setminus\{0\}$ is fixed.
 Let $0<\alpha\leq\beta$, $\Omega\subseteq \R^d$ open and bounded, and consider $\Lp{\infty}$-matrices as multiplication type operators $\Lp{\infty}(\Omega)^{d\times d}\subseteq \Lb(\Lp{2}(\Omega)^{d})$, i.e., linear bounded operators on $\Lp{2}(\Omega)^{d}$ via
$\Lp{2}(\Omega)^{d}\ni f\mapsto a(\cdot)f(\cdot)$ for $a\in \Lp{\infty}(\Omega)^{d\times d}$. Define
\[
   M(\alpha,\beta;\Omega)\coloneqq \{ a\in \Lp{\infty}(\Omega)^{d\times d}: \Re a\geq \alpha, \Re a^{-1}\geq 1/\beta\},
\]i.e., for a.e.\ $x\in\Omega$ we have $\Re\langle y, a(x)y\rangle_{\C^d}\geq\alpha\lvert y\rvert^2_{\C^d}$ for all $y\in\C^d$, and similarly for $\Re a^{-1}\geq 1/\beta$. A sequence $(a_n)_{n\in\N}$ in $M(\alpha,\beta;\Omega)$ is said to \emph{(locally) $\Htopo$-converge} to $a\in M(\alpha,\beta;\Omega)$ if the following holds: For each $f\in \Htopo^{-1}(\Omega) \coloneqq \cH^1(\Omega)'\coloneqq\Lb(\cH^1(\Omega),\C)$, i.e., each continuous functional on the space of $\sobH^1$-functions vanishing at the boundary, let $u_n\in \cH^1(\Omega)$, $n\in \N$, be the Lax--Milgram solution of
\[
    \forall \phi\in \cH^1(\Omega):\langle a_n \cgrad u_n , \cgrad \phi\rangle_{\Lp{2}(\Omega)^d} = f(\phi)\text{,}\]
where $\cgrad$ is the restriction to $\cH^1(\Omega)$ of the weak gradient $\grad$ on $\sobH^1(\Omega)$. 
Then, $u_n\rightharpoonup u \in \cH^1(\Omega)$ and $a_n \grad u_n \rightharpoonup a\grad u \in \Lp{2}(\Omega)^d$ where $u\in \cH^1(\Omega)$ is the unique solution of
\[
    \forall \phi\in \cH^1(\Omega):\langle a \cgrad u , \cgrad \phi\rangle_{\Lp{2}(\Omega)^d} = f(\phi)\text{,}
    \]and $\rightharpoonup$ stands for weak convergence in $\cH^1(\Omega)$ and $\Lp{2}(\Omega)^{d}$ respectively.
    
    This notion has been generalised in \cite{Wa18}, additionally allowing for general bounded linear operators $a \in \Lb(\Lp{2}(\Omega)^d)$, i.e., removing the restriction to multiplication type operators. The corresponding notion coined in \cite{Wa18} is \emph{nonlocal $\Htopo$-convergence}. We quickly recall it in the setting of the main example case presented in \cite{Wa18}: Let $\Omega\subseteq \R^3$ be a bounded weak Lipschitz domain with connected complement. Set
    \[
   \mathcal{M}(\alpha,\beta;\Lp{2}(\Omega)^{3})\coloneqq \{ a\in \Lb(\Lp{2}(\Omega)^{3}): \Re a\geq \alpha, \Re a^{-1}\geq 1/\beta\}\text{,}
\]
where the inequalities are meant in the sense of positive definiteness and $\Re a \coloneqq \frac12(a+a^*)$.
Furthermore, define $\curl \colon \dom(\curl)\subseteq \Lp{2}(\Omega)^3 \to \Lp{2}(\Omega)^3$ via $\curl E\coloneqq \nabla \times E$ for
\[
E\in \dom(\curl)\coloneqq \{ E \in \Lp{2}(\Omega)^3: \nabla \times E \in \Lp{2}(\Omega)^3\},\]
 where $\nabla \times $ is applied in the distributional sense. Note that $\curl$ is then a densely defined, closed linear operator, and $V_{\curl}\coloneqq \dom(\curl)\cap \ker(\curl)^\perp$, endowed with the graph inner product of $\curl$, is a Hilbert space. Furthermore, we obtain the Helmholtz decomposition
   \begin{equation}\label{eq:Helmholtz}
   \Lp{2}(\Omega)^{3}= \ran(\cgrad)\oplus\ran(\curl).
   \end{equation}
    See, e.g., \cite[Section~2]{Wa18} or \cite[Section~1.4]{Bu25} for details.
 A sequence $(a_n)_{n\in\N}$ in $\mathcal{M}(\alpha,\beta;\Lp{2}(\Omega)^{3})$ is said to \textbf{nonlocally $\Htopo$-converge} to some $a\in    \mathcal{M}(\alpha,\beta;\Lp{2}(\Omega)^{3})$ if the following holds: 
    For each $f\in \sobH^{-1}(\Omega)$ and $g\in (V_{\curl})^{\prime}$, let $u_n\in \cH^1(\Omega)$ and $v_n\in V_{\curl}$, $n\in \N$, be the Lax--Milgram solution of
\[
    \forall \phi\in \cH^1(\Omega) : \langle a_n \cgrad u_n , \cgrad \phi\rangle_{\Lp{2}(\Omega)^3} = f(\phi)
\]
and
\[
    \forall \psi \in V_{\curl} : \langle a_n^{-1} \curl v_n , \curl \psi\rangle_{\Lp{2}(\Omega)^3} = g(\psi)
\]respectively. (Note that the assumptions on $\Omega$ render the second problem well-defined, see, e.g., \cite[Section~2]{Wa18} or \cite{TrWa14} or also \Cref{thm:wp} for the details.) Then, $u_n\rightharpoonup u \in \cH^1(\Omega)$ and $a_n \cgrad u_n \rightharpoonup a\cgrad u \in \Lp{2}(\Omega)^d$, where $u\in \cH^1(\Omega)$ is the unique solution of
\[
   \forall \phi\in \cH^1(\Omega): \langle a \cgrad u , \cgrad \phi\rangle_{\Lp{2}(\Omega)^3} = f(\phi),
    \]
    as well as
    $v_n\rightharpoonup v \in V_{\curl}$ and $a_n^{-1} \curl v_n \rightharpoonup a^{-1}\curl v \in \Lp{2}(\Omega)^3$, where $v \in V_{\curl}$ is the unique solution of
\begin{equation}\label{eq:ComplCurlProbl}
    \forall \psi\in V_{\curl}:\langle a^{-1} \curl v , \curl \psi\rangle_{\Lp{2}(\Omega)^3} = g(\psi)\text{.}
    \end{equation}
    Both $\Htopo$-convergence and nonlocal $\Htopo$-convergence induce topologies $\tau_{\Htopo}$ and $\tau_{\nlHtopo}$. The fundamental observations (see, e.g., \cite{MuTa97}, \cite[Chapter~6]{Ta09} and \cite[Section~5]{Wa18}) in this context are that both $(M(\alpha,\beta;\Omega),\tau_{\Htopo})$ and $(\mathcal{M}(\alpha,\beta;\Lp{2}(\Omega)^{3}),\tau_{\nlHtopo})$ form a metrisable and compact topological space. The obvious continuous embedding $(M(\alpha,\beta;\Omega),\tau_{\nlHtopo})\hookrightarrow (M(\alpha,\beta;\Omega),\tau_{\Htopo})$, where $\tau_{\nlHtopo}$ now stands for the trace topology, leads to the question whether or not    $M(\alpha,\beta;\Omega)$ is closed in $(\mathcal{M}(\alpha,\beta;\Lp{2}(\Omega)^{3}),\tau_{\nlHtopo})$. Indeed, in the affirmative case, one has that
    \[(M(\alpha,\beta;\Omega),\tau_{\nlHtopo})= (M(\alpha,\beta;\Omega),\tau_{\Htopo}).\]
    Thus, given local operators only, $\Htopo$-convergence would yield the corresponding convergence for the $\curl$-problems for free including a handy formula for the limit coefficient: simply the inverse of the (local) $\Htopo$-limit of the sequence of inverses.
    
    In \cite[Theorem 5.11]{Wa18}, such a result was claimed. However, the provided proof contains a gap. Skimming the literature for corresponding results, it seems that only \cite[Chapter~5]{ZKO94} provides an affirmative answer for the case of self-adjoint coefficients. Note that the respective proof is based on variational integrals and minimisation techniques invoking $\Gamma$-convergence and rendering the assumption of self-adjointness an essential requirement to make the arguments work. Even though \cite[Theorem 5.11]{Wa18} has already been explicitly used in several publications, e.g., in  \cite{Wa25,BuSkWa24,BFSW24}, and has been confirmed numerically as well as analytically for periodic set-ups, see, e.g., implicitly in \cite{BFSW24,We01},  a proper proof in the non-sefadjoint case appears to have been missing in the literature. 
Based on \cite[Chapter~5]{ZKO94}, a first approach that closes the gap in the proof of \cite[Theorem 5.11]{Wa18} was found in \cite[Section~1.6]{Bu25}. There $d=3$ and topological assumptions on $\Omega$ play central roles. Note that similar techniques have been employed in \cite{BW26,BBCEJW26} transitioning $H$-convergence type statements to statements for nonlocal $H$-convergence in very specific cases. 
    
    Irrespective of completing the theory for homogenisation problems with time-independent partial differential equations, we stress that it was demonstrated in many publications -- we mention \cite{BuSkWa24,BFSW24,SeTrWa22,Wa16a} as example cases -- that nonlocal $\Htopo$-convergence is the decisive tool to obtain homogenisation results for time-dependent partial differential equations; particularly in the framework of evolutionary equations.
    
    In this paper, we shall not only -- like \cite[Section~1.6]{Bu25} and \cite{BW26} -- close the gap in the proof of \cite[Theorem 5.11]{Wa18}, but we will also provide corresponding results for space dimensions different from $3$ and for a more general class of differential operators with different boundary conditions. We neither require smoothness assumptions for the underlying domain nor do we need topological assumptions like, e.g., having a connected complement. In order to present the main result, we recall the necessary abstract notion introduced in \cite{Wa25}; see also \cite{NiWa22}. For this, let $\mathcal{H}$ be a Hilbert space, $\mathcal{H}_0\subseteq \mathcal{H}$ a closed subspace, and set $\mathcal{H}_1\coloneqq \mathcal{H}_0^\perp$. For $a\in \Lb(\mathcal{H})$, we define $a_{jk} \coloneqq \iota_j^* a \iota_k$, where $\iota_j \colon \mathcal{H}_j \hookrightarrow \mathcal{H}$ is the canonical embedding and $j,k\in \{0,1\}$. On
    \begin{equation*}
    \mathcal{M}(\mathcal{H}_0,\mathcal{H}_1)\coloneqq \{a\in \Lb(\mathcal{H}) ; a^{-1} \in \Lb(\mathcal{H}), a_{00}^{-1} \in \Lb(\mathcal{H}_0)\}\text{,}
    \end{equation*}
     we introduce $\tau(\mathcal{H}_0,\mathcal{H}_1)$, the \textbf{Schur topology (w.r.t.~$(\mathcal{H}_0,\mathcal{H}_1)$)}, as the initial Hausdorff topology induced by the four mappings
    \[
      a\mapsto a_{00}^{-1},\       a\mapsto a_{00}^{-1}a_{01},\       a\mapsto a_{10}a_{00}^{-1},\       a\mapsto a_{11}-a_{10}a_{00}^{-1}a_{01},
    \]where the range spaces are endowed with their corresponding weak operator topology.
    
    With $g_0(\Omega) \coloneqq\ran(\cgrad) =\grad[\cH^1(\Omega)]\subseteq \Lp{2}(\Omega)^d$ (which, for bounded $\Omega$, is always closed by Poincar\'e's inequality), the main result may be written as follows.
    \begin{theorem}\label{thm:main-proto} Let $\Omega\subseteq \R^d$ be open and bounded, and $0<\alpha\leq\beta$. Then,
    \[(M(\alpha,\beta; \Omega),\tau_{\Htopo})=(M(\alpha,\beta; \Omega),\tau(g_0(\Omega),g_0(\Omega)^{\perp}))\text{.}\]
     In particular, for $(a_n)_{n\in\N}$ in $M(\alpha,\beta; \Omega)$ and $a\in \Lb(\Lp{2}(\Omega)^d)$, the following conditions are equivalent:
    \begin{enumerate}
      \item[(i)] $a\in M(\alpha,\beta;\Omega)$ and $a_n\to a$ in $\tau_{\Htopo}$;
      \item[(ii)] $a \in \mathcal{M}(g_0(\Omega),g_0(\Omega)^{\perp})$ and $a_n\to a$ in $\tau(g_0(\Omega),g_0(\Omega)^{\bot})$.
    \end{enumerate}    
    \end{theorem}
In fact, we will prove this result for a more general class of differential operators with different boundary conditions, see \Cref{thm:main-general}.    
In particular, \cite[Theorem 5.11]{Wa18} follows:
      \begin{corollary}\label{cor:main-proto} Let $\Omega\subseteq \R^3$ be an open and bounded weak Lipschitz domain with connected complement, and  $0<\alpha\leq\beta$.
      Then,
    \[(M(\alpha,\beta; \Omega),\tau_{\Htopo})=(M(\alpha,\beta; \Omega),\tau_{\nlHtopo})\text{.}\]
     In particular, for $(a_n)_{n\in\N}$ in $M(\alpha,\beta; \Omega)$ and $a\in \Lb(\Lp{2}(\Omega)^3)$, the following conditions are equivalent:
    \begin{enumerate}
      \item[(i)] $a\in M(\alpha,\beta;\Omega)$ and $a_n\to a$ in $\tau_{\Htopo}$;
      \item[(ii)] $a \in \mathcal{M}(\alpha,\beta;\Lp{2}(\Omega)^3)$ and $a_n\to a$ in $\tau_{\nlHtopo}$.
    \end{enumerate}    
    \end{corollary}
    
    As a another application, we provide the consequences for the three-dimensional setting more directly and with less regularity conditions on $\Omega$. In fact, this corollary is the prototypical incarnation of the Free Lunch Theorem. As it asserts an equivalence, both implications can be viewed as individual Free Lunch Theorems.
    \begin{corollary}\label{cor:curl} Let $\Omega\subseteq \R^3$ be an open and bounded weak Lipschitz domain, and $0<\alpha\leq\beta$. Let $(a_n)_{n\in\N}$ in $M(\alpha,\beta;\Omega)$ and $a\in M(\alpha,\beta;\Omega)$. 
    Then, the following conditions are equivalent:
    \begin{enumerate}
    \item[(i)] $a_n\to a$ in $\tau_{\Htopo}$
    \item[(ii)] for all $g\in (V_{\curl})^\prime$ and $v_n\in V_{\curl}$ given as the unique solutions to
    \[
        \forall \psi\in V_{\curl}:\langle a_n^{-1} \curl v_n , \curl \psi\rangle_{\Lp{2}(\Omega)^3} = g(\psi),
    \]
    we have
    \[
        v_n \rightharpoonup v\in V_{\curl}\text{ and } a_n^{-1}\curl v_n \rightharpoonup a^{-1}\curl v \in \Lp{2}(\Omega)^3,
    \]
    where $v\in V_{\curl}$ is the unique solution to
    \[
          \forall \psi\in V_{\curl}:\langle a^{-1} \curl v , \curl \psi\rangle_{\Lp{2}(\Omega)^3} = g(\psi).
    \]
    \end{enumerate}
    \end{corollary}
    As a side note, we mention that the (well-known) independence of the boundary conditions of $\Htopo$-convergence directly implies the independence of the boundary conditions also for the curl-problem provided in (ii) of the previous corollary.
    
Finally, we will apply the Free Lunch Theorem to show a homogenisation result for a fourth-order elliptic problem corresponding to the biharmonic equation (compare \cite{PZ20} and \cite{W17_DCL}). As a consequence, we obtain H-compactness for the biharmonic equation with variable coefficients even though the standard techniques are not available anymore due to the missing (elementary) product rule in this context.
    
    In the following we shall approach a proof for  \Cref{thm:main-proto}  in a slightly more general setting. For this, we introduce a certain class of constant coefficient first-order differential operators in the next section. We prove a  div-curl-type result and provide some explicit examples, where the key assumptions are satisfied. The subsequent section is devoted to introducing a generalisation of (local) $\Htopo$-convergence. There, some elementary properties like compactness are shown under milder conditions than in the original literature. The subsequent section is devoted to the connection to nonlocal $\Htopo$-convergence providing proofs of the results mentioned in this introduction. In the subsequent section, we show the mentioned fourth-order homogenisation result.  The last section contains a conclusion of our findings.

\section{A div-curl Type Result for Linear First-Order Constant Coefficient Differential Operators}\label{sec:dct}

For the remaining section, we assume that $\Omega\subseteq \R^d$ is open and bounded.

Let $A_j \in \C^{k\times \ell}$ for $j\in \{1,\ldots, d\}$ where $k,\ell\in \N\setminus\{0\}$. Consider the differential expression
\begin{equation}\label{eq:DefinMathCalD}
    \mathcal{D}\coloneqq \sum_{j=1}^d \partial_j A_j,
\end{equation}
which formally acts as a first-order differential operator between vector fields with $\ell$ components and vector fields with $k$ components. We define 
\[
    D_c \colon\begin{cases}
\hfill \Cc(\Omega)^\ell \subseteq \Lp{2}(\Omega)^\ell &\to \Lp{2}(\Omega)^k\\
\hfill \phi&\mapsto \sum_{j=1}^d \partial_j A_j\phi
\end{cases}\text{.}
\]
Then, $D_c$ is a densely defined (linear) operator. Moreover, note that for $\psi\in  \Cc(\Omega)^k$ integration by parts for $\phi\in C_c^\infty(\Omega)^\ell$ shows that
\[
   \langle D_c \phi, \psi\rangle_{\Lp{2}(\Omega)^k} =    \langle \sum_{j=1}^d \partial_j A_j\phi , \psi\rangle_{\Lp{2}(\Omega)^k} = \langle \phi,  \sum_{j=1}^d -\partial_j A_j^* \psi\rangle_{\Lp{2}(\Omega)^\ell}.
\]
Thus, via writing
\[
    \mathcal{E}\coloneqq \sum_{j=1}^d \partial_j A^{\ast}_j
\]
for the negative formal adjoint differential operator of $\mathcal{D}$ acting between vector fields with $k$ components and vector fields with $\ell$ components, the adjoint of $D_c$ is a densely defined closed operator, which implies that $D_c$ is closable. Similarly, we introduce
\[
    E_c \colon\begin{cases}
\hfill \Cc(\Omega)^k \subseteq \Lp{2}(\Omega)^k &\to \Lp{2}(\Omega)^\ell\\
\hfill \phi&\mapsto \sum_{j=1}^d \partial_j A_j^*\phi
\end{cases}\text{.}
\]
Next, we define $D_0 \coloneqq \overline{D_c}$ (corresponding to $\mathcal{D}$ with homogeneous boundary conditions) as the operator closure of $D_c$, and
$D_0\subseteq D\coloneqq -E_c^\ast$ (corresponding to $\mathcal{D}$ on its maximal $\Lp{2}$-domain) as the densely defined and closed negative adjoint of $E_c$. Note that $\dom(D_0)$ and $\dom(D)$ endowed with the respective graph inner product are then, by standard arguments, Hilbert spaces.
\begin{example}\label{ex:int}
\begin{enumerate}
 \item[(a)] Let $k=d$, $\ell=1$, $A_j = (\delta_{jk})_{k\in \{1,\ldots,d\}}$. Then, $D_0 =\cgrad$ with $\dom(D_0)=\cH^1(\Omega)$ and $D=\grad$ with $\dom(D)=\sobH^1(\Omega)$.

\item[(b)] Let $d=k=\ell=3$, 
\[
A_1 = \begin{pmatrix}  0 & 0 & 0 \\ 0 & 0 & -1 \\ 0 & 1 & 0\end{pmatrix}, \ A_2 = \begin{pmatrix}  0 & 0 & 1 \\ 0 & 0 & 0 \\ -1 & 0 & 0\end{pmatrix}, \  A_3 = \begin{pmatrix}  0 & -1 & 0 \\ 1 & 0 & 0 \\ 0 & 0 & 0\end{pmatrix}.
\]
Then, $\ccurl\coloneqq D_0 \subseteq \curl=D$ with 
\[
\dom(D_0) =\{ E \in \dom(\curl); \exists (E_n)_{n\in\N} \text{ in }\Cc(\Omega)^3\colon E_n \to E,\ \curl E_n \to \curl E \}.
\]
\item[(c)] Define $\Lp{2,\sym}(\Omega)^{d \times d}\coloneqq \{ \Psi \in \Lp{2}(\Omega)^{d\times d}: \Psi =\Psi^\top\}$ and consider
\[
    \sym\cgrad \colon \begin{cases}
\hfill\cH^1(\Omega)^d \subseteq \Lp{2}(\Omega)^d &\to \Lp{2,\sym}(\Omega)^{d \times d}\\
\hfill \phi&\mapsto \frac12(\partial_j\phi_k+\partial_k\phi_j)_{j,k\in \{1,\ldots,d\}}
\end{cases}\text{,}
\]
as well as
\[
    \sym\grad \colon \begin{cases}
\hfill\sobH^1(\Omega)^d \subseteq \Lp{2}(\Omega)^d &\to \Lp{2,\sym}(\Omega)^{d \times d}\\
\hfill \phi&\mapsto \frac12(\partial_j\phi_k+\partial_k\phi_j)_{j,k\in \{1,\ldots,d\}}
\end{cases}\text{.}
\]
By Korn's first and second inequality $\sym\cgrad$ and $\sym\grad$, the \textbf{symmetrised gradient} with and without homogeneous boundary conditions, are densely defined, closed linear operators. For the second one, we require sufficient regularity of $\partial\Omega$, e.g., $\Omega$ being a strong Lipschitz domain (i.e., $\partial\Omega$ is locally a Lipschitz graph) is sufficient. Note that, by ignoring the entries in the lower left corner, $\Lp{2,\sym}(\Omega)^{d \times d}$ is isomorphic to $\Lp{2}(\Omega)^{d(d+1)/2}$ and, moreover, the symmetrised gradient (with and without homogeneous boundary conditions) can be rewritten as
\[
    \sym\grad \phi \sim\begin{pmatrix}\partial_1\phi_1&  \frac12(\partial_{2}\phi_1+\partial_{1}\phi_{2}) & \ldots & \frac12(\partial_{d}\phi_1+\partial_{1}\phi_{d}) & \partial_2\phi_2 & \ldots& \partial_{d}\phi_{d}\end{pmatrix}^\top.
\]
Thus, taking $\ell=d$, $k = d(d+1)/2$, appropriate choices for the $A_j$, and discussing the appearing $\dom(D)$ (see, e.g., \cite[Remark~2.4]{BD24}), we find that $\sym\cgrad$ and $\sym\grad$ are once more special cases of such $D_0$ and $D$ from above. For illustration, we provide the matrices in the case $d=3$:
\[
    A_1 = \begin{pmatrix} 1     &      0    &   0               \\
                                       0       & \frac12 &   0     \\
                                       0     &     0        & \frac12 \\        
                                       0    & 0         &     0     \\
                                        0    & 0         &     0     \\
                                         0    & 0         &     0      \end{pmatrix}, 
  A_2 = \begin{pmatrix} 0     &      0    &   0               \\
                                       \frac12       & 0 &   0     \\
                                       0     &     0        & 0 \\        
                                       0    & 1         &     0     \\
                                        0    & 0         &     \frac12      \\
                                         0    & 0         &     0      \end{pmatrix},
   A_3 = \begin{pmatrix} 0     &      0    &   0               \\
                                       0     & 0 &   0     \\
                                       \frac12      &     0        & 0\\        
                                       0    &  0         &     0     \\
                                        0    &  \frac12          &     0    \\
                                         0    & 0         &     1     \end{pmatrix}.                                         \qedhere
\]
\end{enumerate}
\end{example}

Our aim of this section is to establish a div-curl lemma type result; see, e.g., \cite[Chapter~7]{Ta09} or \cite{BCM09,W17_DCL,Pauly2019} for more information. The following arguments are less involved and are rooted in the following elementary fact.

\begin{lemma}[{{\cite[Lemma 14.4.6]{SeTrWa22}}}]\label{lem:dc0} Let $H$ be a Hilbert space, $(q_n)_{n\in\N}$, $(r_n)_{n\in\N}$ weakly convergent to $q$ and $r$ respectively in $H$, $X\subseteq H$ a closed subspace, and $\iota\colon X\hookrightarrow H$ the canonical embedding. If $q_n\in X$ for all $n\in\N$ and $\iota^* r_n\to \iota^* r$ strongly in $X$, then
\[
   \lim_{n\to\infty}\langle q_n, r_n\rangle_H = \langle q, r\rangle_H.
\]
\end{lemma}
\begin{proof}
This is an immediate consequence of $\langle q_n, r_n\rangle_H=\langle q_n, \iota^\ast r_n\rangle_X$ for $n\in\N$.
\end{proof}
For the remaining section, we fix $\mathcal{D}$ and let $C$ be any closed operator with $D_0\subseteq C\subseteq D$.
To establish our div-curl version, we require another assumption on $C$:
\begin{enumerate}
 \item[\textbf{(com)}] $V_C\coloneqq\dom(C)\cap \ker(C)^\perp$ as a Hilbert space is compactly embedded into $\Lp{2}(\Omega)^\ell$.
\end{enumerate}
\begin{remark}\label{rem:CompactInvC}
Note that it is then elementary to show (see, e.g., the FA-toolbox in \cite{PZ20}), that $\ran(C)\subseteq L_2(\Omega)^k$ is closed (and thus weakly closed). Hence,
\[\iota_0^\ast \widetilde{C\kappa_0}\colon V_C\to \ran(C)\]
 is bounded and boundedly invertible, where $\iota_0 \colon  \ran(C)\hookrightarrow\Lp{2}(\Omega)^k$ and $\kappa_0 \colon \ker(C)^\bot \hookrightarrow \Lp{2}(\Omega)^\ell$ are the canonical embeddings and $\widetilde{C\kappa_0}\in\Lb(V_C, \Lp{2}(\Omega)^k)$ is defined as $\widetilde{C\kappa_0}v\coloneqq Cv$ for $v\in V_C$. In particular, the
 inverse of
 \[\iota_0^\ast C\kappa_0\colon \dom(C)\cap \ker(C)^\perp\subseteq \Lp{2}(\Omega)^\ell\to \ran(C),\]
  admitting values in $\dom(C)\cap\ker(C)^\perp\subseteq \Lp{2}(\Omega)^\ell$, is a compact operator.
\end{remark}
We will use the notation from \Cref{rem:CompactInvC} in the remaining paper.
We now review the operators from \Cref{ex:int}.
\begin{example}\label{ex:comp}
\begin{enumerate}
 \item[(a)] If $\cgrad\subseteq C\subseteq \grad$ and $\Omega$ satisfies the cone condition, condition (com) is implied by Rellich's selection theorem. Note that we do not need the cone condition in the case $C=\cgrad$.

\item[(b)] For $d=3$, if either $\ccurl=C$ or $C=\curl$, and if $\Omega$ additionally is a weak Lipschitz domain, (com) is established by the Picard--Weber--Weck selection theorem, see, e.g., \cite{Pi84}. Other realisations $\ccurl\subseteq C\subseteq \curl$ satisfy (com), if $C$ carries mixed Neumann--Dirichlet boundary conditions and $\Omega$ is weak Lipschitz, see \cite{BPS19}. We refer to \cite{PS23} for a related result for strong Lipschitz domains.

\item[(c)] For $\sym\cgrad\subseteq C\subseteq\sym\grad$ and $\Omega$ being a strong Lipschitz domain to allow for Korn's second inequality, (com) again follows from Rellich's selection theorem. There is no regularity requirement for $\Omega$, if $C=\sym\cgrad$.\qedhere
\end{enumerate}
\end{example}

A generalisation of \cite[Theorem 14.4.7]{SeTrWa22} now reads as follows.
\begin{theorem}\label{lem:dceasy} Assume condition (com). Let $(q_n)_{n\in\N}, (r_n)_{n\in\N}$ in $\Lp{2}(\Omega)^k$ be weakly convergent to $q$ and $r$ respectively. Further, assume that 
\[
    \iota_0^* r_n\to \iota_0^* r
\] strongly in $\ran(C)$ as $n\to\infty$.

If $q_n\in {\ran}(C)$ for all $n\in \N$, then
\begin{equation}\label{eq:conv}
    \forall\phi\in \Cc(\Omega):\int_\Omega \langle r_n(x),q_n(x)\rangle_{\C^k} \phi(x) \mathrm{d}x \to     \int_\Omega \langle r(x),q(x)\rangle_{\C^k} \phi(x) \mathrm{d}x\text{.}
\end{equation}
\end{theorem}
\begin{proof}
Define $v_n\coloneqq (\iota_0^\ast C\kappa_0 )^{-1}q_n \in \dom(C)\cap \ker(C)^\perp\subseteq\Lp{2}(\Omega)^\ell$ for $n\in\N$, where $\kappa_0 \colon \ker(C)^\bot \hookrightarrow \Lp{2}(\Omega)^\ell$ is the canonical embedding. By \Cref{rem:CompactInvC}, $v_n\to v\coloneqq (\iota_0^\ast C\kappa_0 )^{-1}q$ strongly in $\ker(C)^\perp\subseteq\Lp{2}(\Omega)^\ell$ and weakly in $\dom(C)$. Thus, we compute
 \begin{align*}
 \int_\Omega \langle r_n(x),q_n(x)\rangle_{\C^k} \phi(x) \mathrm{d} x & = \langle r_n, \phi q_n\rangle_{\Lp{2}(\Omega)^k} \\
 & = \langle r_n, \phi C v_n\rangle_{\Lp{2}(\Omega)^k} \\
 & = \langle r_n, C (\phi v_n)\rangle_{\Lp{2}(\Omega)^k} -    \langle r_n,  \sum_{j=1}^d (\partial_j \phi) A_j  v_n\rangle_{\Lp{2}(\Omega)^k} \\
 & \to \langle r, C (\phi v)\rangle_{\Lp{2}(\Omega)^k} -  \langle r,  \sum_{j=1}^d (\partial_j \phi) A_j  v\rangle_{\Lp{2}(\Omega)^k}\\
 & =  \langle r, \phi q\rangle_{\Lp{2}(\Omega)^k}= \int_\Omega \langle r(x),q(x)\rangle_{\C^k} \phi(x) \mathrm{d}x,
 \end{align*}
 where we used $C (\phi v_n)\rightharpoonup C (\phi v)$ weakly in $\Lp{2}(\Omega)^k$ and \Cref{lem:dc0} for the limit.
\end{proof}

\section{$\Htopo$-Convergence for $C$}\label{sec:HConForC}

In this section, we introduce a notion of (local) $\Htopo$-convergence for the setting introduced in the previous section, i.e., we will once again assume that $\Omega\subseteq \R^d$ is open and bounded, we will fix some $\mathcal{D}$ as in~\eqref{eq:DefinMathCalD}, some closed $D_0\subseteq C\subseteq D$, and we will assume (com) throughout the section.

First, we recall the following well-posedness theorem from \cite{TrWa14}. Note that the theorem as presented in the following can also be proven with an appropriate version of the Lax--Milgram lemma. For that matter, we are first required to recall a standard result concerning the adjoint of $C$. Again, this can be found in, e.g., \cite{TrWa14} or as part of the results in the FA-toolbox in \cite{PZ20} (see also \cite{BW26} for some more specific details).
 We define the antilinear and bounded
 \[
 \bigl(\widetilde{C\kappa_0}\bigr)^\diamond \colon
 \begin{cases}
\hfill\Lp{2}(\Omega)^k &\to  V_C^\prime\\
\hfill q&\mapsto \Bigl(V_C\ni \phi \mapsto \langle q,C \phi\rangle_{\Lp{2}(\Omega)^k}\Bigr)
\end{cases}\text{.}
\]
for $q\in \Lp{2}(\Omega)^k$.
Then, for $a\in \Lb(\Lp{2}(\Omega)^k)$, $u\in\dom(C)$, and $f\in V_C^\prime$, we have
\begin{equation}\label{eq:ellipticWPThmProof}
   \Big(\forall\phi\in V_C: \langle a Cu , C \phi\rangle_{L_2(\Omega)^k} = f(\phi)\Big) \iff \bigl(\widetilde{C\kappa_0} \bigr)^\diamond a C u = f.
\end{equation}
Using $\Phi(\widetilde{C\kappa_0})^\diamond= (\widetilde{C\kappa_0})^\ast$ (see \cite[Proposition 1.2.3]{BW26}) with the usual antilinear Riesz isomorphism $\Phi\colon V_C^\prime\to V_C$, we easily see
\[\ran(C) = \ker\bigl( \bigl(\widetilde{C\kappa_0}\bigr)^\diamond\bigr)^\perp\text{ and }\ran \bigl( \bigl(\widetilde{C\kappa_0}\bigr)^\diamond\bigr)=V_C^\prime\text{.}\]
Thus, we obtain that
\[
\bigl(\widetilde{C\kappa_0} \bigr)^\diamond\iota_0\colon\ran(C)\to V_C^\prime
\] has an antilinear and bounded inverse.
Hence, we can continue~\eqref{eq:ellipticWPThmProof}
\[
\bigl(\widetilde{C\kappa_0} \bigr)^\diamond a C u = f \iff \bigl(\widetilde{C\kappa_0} \bigr)^\diamond \iota_0\iota_0^\ast a C u = f 
\iff  \iota_0^* a \iota_0\iota_0^\ast C u = \bigl(\bigl(\widetilde{C\kappa_0} \bigr)^\diamond\iota_0\bigr)^{-1} f\text{,}
\]
which altogether yields the following result, which also introduces the \textbf{Lax--Milgram operator}, $\operatorname{LM}_{C}(a)$, associated to an abstract divergence form problem involving the operator $C$ and (operator) coefficient $a$.
\begin{theorem}[{{\cite[Theorem 3.1]{TrWa14}}}]\label{thm:wp}
Let $a\in \Lb(\Lp{2}(\Omega)^k)$ with $\iota_0^\ast a\iota_0$ boundedly invertible.

If $f\in V_C^\prime$, then there exists a uniquely determined $u\in V_C$ such that
\begin{equation}\label{eq:AbsElliPDE}
    \forall \phi\in V_C:\langle a Cu, C\phi\rangle_{\Lp{2}(\Omega)^k} = f(\phi).
 \end{equation}
To be precise, we have $u= \operatorname{LM}_{C}(a) f$, where
\[\Lb(V_C^\prime,V_C)\ni \operatorname{LM}_{C}(a)\coloneqq (\iota_0^\ast \widetilde{C\kappa_0})^{-1}(\iota_0^\ast a\iota_0)^{-1}\bigl(\bigl(\widetilde{C\kappa_0} \bigr)^\diamond\iota_0\bigr)^{-1}\text{,}\]
where
\[\iota_0^\ast \widetilde{C\kappa_0} \colon V_C\to \ran(C)\quad\text{ and }\quad\bigl(\widetilde{C\kappa_0} \bigr)^\diamond\iota_0\colon\ran(C)\to V_C^\prime
\]
are (anti-)linear,
bounded, and have (anti-)linear and bounded inverses.
\end{theorem}

In the following we will simply write $V$ instead of $V_{C}$.

We are now in the position to introduce $\Htopo$-convergence in the present setting. For $0<\alpha\leq\beta$, define
\[
   M(\alpha,\beta;C)\coloneqq \{a\in \Lp{\infty}(\Omega)^{k\times k}: \Re a\geq \alpha, \Re a^{-1}\geq 1/\beta\}.
\]
Let $(a_n)_{n\in\N}$, $a$ in $M(\alpha,\beta;C)$. Then, $a_n$ \textbf{$\Htopo$-converges}  to $a$, if and only if the following is satisfied:

For all $f\in V'$ and $n\in \N$, let $u_n \in V$ be the unique solution of
\[
    \forall \phi\in V:\langle a_n Cu_n, C \phi\rangle_{\Lp{2}(\Omega)^k} = f(\phi).
    \]
    Then $u_n\rightharpoonup u$ weakly in $V$ and $a_n C u_n \rightharpoonup aC u$ weakly in $\Lp{2}(\Omega)^k$, where $u\in V$ satisfies
    \[
    \forall \phi\in V:\langle a C u , C \phi\rangle_{\Lp{2}(\Omega)^k} = f(\phi).
    \]
  Before we discuss the main properties of $\Htopo$-convergence, we need to impose a condition on $\mathcal{D}$.
Namely, we assume a non-degeneracy condition of the following form:
\begin{enumerate}
 \item[\textbf{(nd)}] $\spn_{j\in \{1,\ldots,d\}} \ran A_j = \R^k$.
\end{enumerate}
We revisit \Cref{ex:int} concerning the condition (nd).
\begin{example}\label{exa:nd}
\begin{enumerate}
\item[(a)] The condition (nd) is trivially satisfied for $D=\grad$.
\item[(b)] For $d=3$ and $D=\curl$, (nd) also holds. In fact, already $\spn \ran A_1\cup \ran A_2 = \R^3$.
\item[(c)] For $D=\sym\grad$, (nd) also holds. In the treated example case $d=3$, all three matrices are needed to obtain (nd).\qedhere
\end{enumerate}
\end{example}
We quickly mention the relevant consequence of condition (nd).
\begin{lemma}\label{lem:localvectors} Assume (nd).
 Let $\xi\in \CC^k$, $\omega\subseteq \Omega$ open with $\overline{\omega}\subseteq \Omega$. Then, there exists $u\in \Cc(\Omega)^\ell$ such that $C u = \xi$ on $\omega$.
\end{lemma}
\begin{proof}
By (nd), there exist $\alpha_1,\ldots,\alpha_d \in \CC$ and $q_1,\ldots, q_d \in \CC^k$ such that $q_j = A_j r_j$ for some $r_j \in \CC^\ell$ for all $j\in \{1,\ldots,d\}$ and $\xi = \sum_{j=1}^d \alpha_j q_j$. By appropriately scaling $r_j$, we may assume, without loss of generality, that $\alpha_j =1$ for all $j\in \{1,\ldots,d\}$. For all $j\in \{1,\ldots,d\}$, we find $\phi_j \in \Cc(\Omega)^\ell$ such that, on $\omega$, we have $\phi_j(x)=x_j r_j$. It follows that $u\coloneqq \sum_{j=1}^d \phi_j$ has the desired properties.
\end{proof}
  
  We now get to the main structural theorem regarding $\Htopo$-convergence. In the proof, for the most parts, we will mainly follow and adapt the arguments in \cite[page~82 and Theorem~6.5]{Ta09}. The concluding part, however, where in \cite{Ta09} the density of piecewise affine functions in $\sobH^1$ is used, we deviate significantly from the proposed proof. In fact, by not just assuming piece-wise affine vector fields to be dense, we obtain a wider applicability of our result, see \Cref{rem:ProofWIthFEMPropAP} below.
\begin{theorem}\label{thm:Hcom} Assume conditions (com) and (nd). Then $\Htopo$-convergence induces a metric topology on $ M(\alpha,\beta;C)$ denoted by $\tau_{\Htopo}$. $(M(\alpha,\beta;C),\tau_{\Htopo})$ is even (sequentially) compact.
\end{theorem} 

\begin{proof}
A moment's reflection reveals that $\tau_{\Htopo}$ is the initial topology with respect to the mappings
    \begin{equation}\label{eq:LM-reform}
    \begin{aligned}
    M(\alpha,\beta;C) \ni a&\mapsto \operatorname{LM}_{C}(a) \in \Lb(V^\prime,V) \text{ and}\\
    M(\alpha,\beta;C) \ni a&\mapsto a\widetilde{C\kappa_0} \operatorname{LM}_{C}(a) \in \Lb(V^\prime,\Lp{2}(\Omega)^k),
    \end{aligned}
    \end{equation} where both $\Lb(V^\prime,V)$ and $\Lb(V^\prime,\Lp{2}(\Omega)^k)$ are endowed with the weak operator topology.
Since $V^\prime$ is separable\footnote{Simply isometrically embed $V$ into
$\Lp{2}(\Omega)^\ell\times \Lp{2}(\Omega)^k$.}, it suffices to only consider a countable orthonormal basis of $V^\prime$ as choices for $f$.
These correspond to countably many bounded subsets of $V$ and $\Lp{2}(\Omega)^k$ endowed with the respective weak topology, i.e., countably many metric spaces. Thus, $\tau_{\Htopo}$ is induced by a pseudometric. It remains to show that limits of sequences are unique. This immediately follows from \Cref{lem:localvectors}.

 For a given sequence $(a_n)_{n\in\N}$, the Banach--Alaoglu theorem yields a subsequence (that we will still call $(a_n)_{n\in\N}$), $B_0\in \Lb(V^\prime,V)$, and $B_1\in \Lb(V^\prime,\Lp{2}(\Omega)^k)$ such that
  \[\operatorname{LM}_{C}(a_n)\to B_0\quad\text{ and }\quad a_n \widetilde{C\kappa_0} \operatorname{LM}_{C}(a_n)\to B_1\] in the respective weak operator topology. Since $B_0$ is boundedly invertible\footnote{Let $\mathcal{H}$ be a Hilbert space, $c,d>0$, $S,R_n\in\Lb(\mathcal{H})$ with $\Re R_n\geq c$ and $\Re R_n^{-1}\geq d$ for $n\in\N$, and $R_n^{-1}\to S$ in the weak operator topology. Then, easy calculations show $\Re S\geq d$ and $\Re S^{-1}\geq c$.}, we infer
 $B_1 = B_1B_0^{-1}B_0$ and define $B\coloneqq B_1 B_0^{-1}\in \Lb(V,\Lp{2}(\Omega)^k)$.
 This implies 
 \begin{equation}\label{eq:WOTLimitOfFluxesBB0}
 a_n \widetilde{C\kappa_0} \operatorname{LM}_{C}(a_n)\to BB_0
 \end{equation}
 in the weak operator topology and thus
\begin{equation}\label{eq:WeakFormBB0C}
\forall \phi\in V:\langle BB_0 f , C \phi\rangle_{\Lp{2}(\Omega)^k} = f(\phi)
 \end{equation}
for each $f\in V^\prime$.
It remains to show that there is $a\in M(\alpha,\beta;C)$ with $B=a\widetilde{C\kappa_0}$: Indeed, this turns~\eqref{eq:WeakFormBB0C} into
\[\forall \phi\in V:\langle a \widetilde{C\kappa_0}  B_0 f , C \phi\rangle_{\Lp{2}(\Omega)^k} = f(\phi)\]
for each $f\in V^\prime$, i.e., $B_0=\operatorname{LM}_{C}(a)$ and, due to~\eqref{eq:WOTLimitOfFluxesBB0}, also $B_1=a \widetilde{C\kappa_0}\operatorname{LM}_{C}(a)$.

For the proof of $B=a\widetilde{C\kappa_0}$, let $v\in V$ and define $f\coloneqq B_0^{-1}v\in V^\prime$ (trivially extended to $\ker(C)$ by $0$). Then, for $n\in \N$ and $u_n\coloneqq \operatorname{LM}_{C}(a_n)f\in V$, we obtain $u_n\rightharpoonup v$ weakly in $V$, $a_nCu_n\rightharpoonup Bv$ weakly in $\Lp{2}(\Omega)^k$. For $n\in \N$ and for some $\phi\in \Cc{(\Omega)}$ we infer
\begin{equation}\label{eq:fpun}
 f ( \phi u_n) = \langle a_nC u_n, \phi C u_n+\sum_{j=1}^d (\partial_j \phi) A_j u_n\rangle_{\Lp{2}(\Omega)^k}.
\end{equation} By condition (com) we deduce $u_n\to v$ in $\Lp{2}(\Omega)^\ell$, employing the product rule, $\phi u_n \rightharpoonup \phi v$ in $\dom(C)$, and, by~\eqref{eq:fpun}, as $n\to\infty$,
\[
   f ( \phi v)  = \lim_{n\to\infty} \langle a_nC u_n, \phi C u_n\rangle_{\Lp{2}(\Omega)^k}+\langle B v,\sum_{j=1}^d (\partial_j \phi) A_j v\rangle_{\Lp{2}(\Omega)^k}.
\]
Moreover, by~\eqref{eq:WeakFormBB0C}, that is,
 \[ f (\phi v) = \langle B v, \phi C v\rangle_{\Lp{2}(\Omega)^k} + \langle B v, \sum_{j=1}^d (\partial_j \phi) A_j v\rangle_{\Lp{2}(\Omega)^k},\] we get
\[
  \lim_{n\to\infty} \langle a_nC u_n, \phi C u_n\rangle_{\Lp{2}(\Omega)^k} =\langle B v, \phi C v\rangle_{\Lp{2}(\Omega)^k}.
\]
Using the definition of $M(\alpha,\beta;C)$ and the Cauchy--Schwarz inequality, we infer that, for all $\phi\in\Cc{(\Omega)}$ with $\phi\geq 0$ and all $n\in\N$,
\[
\alpha\frac{\lvert\langle\phi^{1/2} Cv,\phi^{1/2} Cu_n\rangle_{\Lp{2}(\Omega)^k}\rvert^2}{\lVert\phi^{1/2} Cv\rVert_{\Lp{2}(\Omega)^k}^2}\leq
\alpha \lVert\phi^{1/2} Cu_n\rVert_{\Lp{2}(\Omega)^k}^2\leq \Re \langle a_nC u_n, \phi C u_n\rangle_{\Lp{2}(\Omega)^k}\text{.}
\]
Letting $n\to\infty$ in the two inner products and recalling $Cu_n\rightharpoonup Cv$ weakly in $\Lp{2}(\Omega)^k$, we obtain
\[
\Re \langle B v, \phi C v\rangle_{\Lp{2}(\Omega)^k}\geq \alpha  \lVert\phi^{1/2} Cv \rVert_{\Lp{2}(\Omega)^k}^2\]
for all $\phi\in\Cc{(\Omega)}$ with $\phi\geq 0$.
Hence,
\begin{equation}\label{eqPointwiseAlphaforB}
\Re \langle B v, C v\rangle_{\C^k} \geq \alpha \lvert C v\rvert_{\C^k}^2
\end{equation}
almost everywhere.
Similarly, we deduce
\begin{equation}\label{eqPointwiseBetaforB}
  \Re \langle B v,C v\rangle_{\C^k} \geq \frac{1}{\beta}\lvert Bv\rvert_{\C^k}^2
\end{equation}
almost everywhere and, by the Cauchy--Schwarz inequality,
\begin{equation}\label{eq:locality}
    \lvert Bv\rvert_{\C^k} \leq \beta \lvert C v\rvert_{\C^k}
\end{equation}
almost everywhere.

For the definition of $a$, take a countable orthonormal basis $(\varphi_n)_{n\in\N}$ of $V$ and a standard exhaustion $(\Omega_n)_{n\in\N}$ of $\Omega$.
By \Cref{lem:localvectors}, for $n\in\mathbb{N}$ and the canonical basis vectors $e_{j}\in \C^k$ for $j\in\{1,\dots,k\}$, we find $\chi_{n,j}\in V$ such that $C \chi_{n,j} = e_j$ on $\Omega_n$.
We define the $\Q+\iu\Q$-linear span
\[
\mathcal{F}\coloneqq \spn_{\Q+\iu\Q}(\{\varphi_n : n\in\N\}\cup \{\chi_{n,j} :n\in\N, j\in\{1,\dots,k\}\}).
\]
This is a countable set.
Taking into account inequality \eqref{eq:locality}, we infer that there exists a null set $N\subseteq \Omega$ such that
\begin{align}
|B\phi(t)|&\leq \beta |C\phi(t)|,\label{eq:localD1}\\
B(\phi+\mu\psi)(t)&=B\phi(t)+\mu B\psi(t),\label{eq:localD2}\\
C(\phi+\mu\psi)(t)&=C\phi(t)+\mu C\psi(t)\label{eq:localD3},
\end{align}
for any $t\in \Omega\setminus N$, $\phi, \psi \in \mathcal{F}$ and $\mu\in \Q+\iu\Q$.
For $t\in \Omega\setminus N$, we define the map $\alpha_t\colon\dom(\alpha_t)\subseteq\C^k \to\C^k$ via
the $\Q+\iu\Q$-linear subspace $\dom(\alpha_t)\coloneqq\{C\phi(t):\phi \in \mathcal{F}\}\subseteq\C^k$ and
\[
\alpha_t C\phi(t)\coloneqq B\phi(t)
\]
for all $\phi \in \mathcal{F}$.
We observe that $\alpha_t$ attains unique values since, if we take $\phi, \psi \in \mathcal{F}$ such that $C\phi(t)=C\psi(t)$, then, by \eqref{eq:localD1}, \eqref{eq:localD2}, and \eqref{eq:localD3}, we have
\begin{equation}\label{eq:LipschitzUnique}
|B\phi(t)- B\psi(t)|=|B(\phi-\psi)(t)|\leq \beta |C(\phi-\psi)(t)|=\beta |C\phi(t)- C\psi(t)|=0.
\end{equation}
Taking $\Omega_n$ such that $t\in\Omega_n$ and choosing $\phi=\mu \chi_{n,j}$ for $j\in\{1,\dots,k\}$ and $\mu\in \Q+\iu\Q$, we deduce that $\Q^k+\iu\Q^k\subseteq\dom(\alpha_t)$.
Moreover, \eqref{eq:localD2} and \eqref{eq:localD3} imply that $\alpha_t$ is $\Q+\iu\Q$-linear, and, by an argument similar to \eqref{eq:LipschitzUnique},
$\alpha_t$ is Lipschitz continuous. This yields a unique continuous extension of $\alpha_t$ to $\C^k$ that, by continuity, is $\C$-linear.
All in all, there exists a unique $a(t)\in \C^{k\times k}$ such that $\alpha_t=a(t)\restriction_{\dom(\alpha_t)}$.
We now claim that the arising (trivially extended) $a\colon \Omega \to \C^{k\times k}$ is an element of $\Lp{\infty}(\Omega)^{k\times k}$.
The boundedness follows by \eqref{eq:localD1} and $e_j\in \dom(\alpha_t)$ for $j\in\{1,\dots,k\}$ and $t\in\Omega\setminus N$. Furthermore, a standard disjointification argument based on $(\Omega_n)_{n\in\N}$ and $\{\chi_{n,j} :n\in\N, j\in\{1,\dots,k\}\}$ shows that $a(\cdot)e_j$ can be represented as a locally finite sum of measurable functions for $j\in\{1,\dots,k\}$. Thus, $a$ is measurable.
Hence, $a\in\Lp{\infty}(\Omega)^{k\times k}$ and $B\phi=aC\phi$ for all $\phi\in \mathcal{F}$. By density of $\spn_{\Q+\iu\Q}\{\varphi_n : n\in\N\}$ in $V$, we conclude that $Bv=a\widetilde{C\kappa_0}v$ for all $v\in V$.

Finally, \Cref{lem:localvectors} applied to~\eqref{eqPointwiseAlphaforB} and~\eqref{eqPointwiseBetaforB} shows that $a\in M(\alpha,\beta;C)$.
\end{proof}
\begin{remark}\label{rem:ProofWIthFEMPropAP}
Note that the definition of $a$ in the proof of \Cref{thm:Hcom} hinges upon an argument that is closely related to disproving the existence of nontrivial solutions to Cauchy's functional equation under certain regularity assumptions. This approach stands in considerable contrast to the original proof of \cite[Theorem~6.5]{Ta09} which applies some additional density property. In our generalised setting, a corresponding approximation property could read
\begin{enumerate}
 \item[\textbf{(ap)}] $A_C\coloneqq\{\phi\in V: C\phi\text{ is piecewise constant}\}$ is dense in the Hilbert space $V$.
\end{enumerate}
Note that density of $\{\phi\in \dom(C): C\phi\text{ is piecewise constant}\}$ in the Hilbert space $\dom(C)$ implies (ap) since $\widetilde{\kappa_0\kappa_0^\ast}\in\Lb(\dom(C),V)$ where $\widetilde{\kappa_0\kappa_0^\ast}\phi\coloneqq\kappa_0\kappa_0^\ast \phi$ for $\phi\in\dom(C)$.

In a certain sense, condition (ap) is a FEM-property that is known for all examples treated in \Cref{ex:int}:
\begin{enumerate}
\item[(a)] For $C=\cgrad$ or $C=\grad$ and $\Omega$ with a sufficiently regular boundary, (ap) follows from the classical piecewise affine FEM-approximation theory. Note that we do not need any boundary regularity in the case $C=\cgrad$.
\item[(b)] For $d=3$, $C=\ccurl$ or $C=\curl$, and $\Omega$ with a sufficiently regular boundary and geometry, (ap) follows from the classical FEM-approximation theory with lowest order Nédélec elements.
\item[(c)] For $C=\sym\cgrad$ or $C=\sym\grad$ and $\Omega$ with a sufficiently regular boundary, (ap) follows from (a) in combination with Korn's inequalities. Once again, we do not need any boundary regularity in case of homogeneous boundary conditions, $C=\sym\cgrad$.
\end{enumerate}
The adapted part of the proof of \Cref{thm:Hcom} concerning the definition of $a$ under the additional assumption of (ap) could now read as follows:

For $n\in\N$ and $j\in\{1,\dots,k\}$, we define $a_i\colon \Omega_n \to \C^{k}$ via $a_i \coloneqq B\chi_{n,j}\restriction_{\Omega_n}$. By \eqref{eq:locality}, this yields a well-defined $a_i\in\Lp{\infty}(\Omega)^{d}$. Hence,
\[
a\colon \begin{cases}
\hfill\Omega &\to \C^{k \times k}\\
\hfill x&\mapsto \left(\xi\mapsto \sum_{i=1}^k\xi_i a_i(x)\right)
\end{cases}
\]
is well-defined and an element of $\Lp{\infty}(\Omega)^{d\times d}$.  Moreover, \eqref{eq:locality} shows that $Bv=a\widetilde{C\kappa_0}v$ for $v\in A$. For general $v\in V$, we obtain the result by (ap).

As we will see in \Cref{sec:NonTrivEx}, the proof without condition (ap) is a relevant improvement that allows to treat differential operators to which currently no FEM-theory exists yet.
\end{remark}
The following lemma is a central tool of this paper.
It is a generalisation of \cite[Theorem~5.2]{ZKO94}, extended to the non-selfadjoint case and more general differential operators.
\begin{theorem}\label{prop:Hconplus} Assume (nd) and (com). Consider $(a_n)_{n\in\N}$, $a$ in $M(\alpha,\beta;C)$ with $a^\ast_n\to a^\ast$ in $\tau_{\Htopo}$, $h\in V^\prime$, and $z\in L_2(\Omega)^k$. 

Let $u_n\in V$ be the unique solution to
\[
      \forall \phi \in V:\langle a_n (C u_n +z), C \phi\rangle_{\Lp{2}(\Omega)^k} = h(\phi).
\]
Then, $u_n\rightharpoonup u$ weakly in $V$ and $a_n (C u_n +z) \rightharpoonup a (C u +z)$ weakly in $\Lp{2}(\Omega)^k$, where $u\in V$ uniquely solves
\[
   \forall\phi \in V:\langle a (C u +z), C \phi\rangle_{\Lp{2}(\Omega)^k} = h(\phi).
\]
\end{theorem}
\begin{proof} Choosing an appropriate subsequence, we obtain $v\in V$ and $q\in \Lp{2}(\Omega)^k$  with $u_n\rightharpoonup v$ weakly in $V$ and $a_n (C u_n +z) \rightharpoonup q$ weakly in $\Lp{2}(\Omega)^k$ since both sequences are bounded. Let $u_0\in V$  and define $v_n\in V$ such that
\[
   \forall\phi \in V:\langle a_n^\ast C v_n, C \phi\rangle_{\Lp{2}(\Omega)^k} =    \langle a^\ast C u_0, C \phi\rangle_{\Lp{2}(\Omega)^k}.
\]
Then, by definition of $\Htopo$-convergence, $v_n\rightharpoonup u_0$ and $a_n^\ast C v_n \rightharpoonup a^\ast C u_0$ in the respective weak topologies. Then, we have, using the embedding $\iota_0\colon \ran(C)\hookrightarrow L_2(\Omega)^k$, that $\iota_0^*(a_n (C u_n +z))$ is independent of $n\in\N$ and so is $\iota_0^*(a_n^* (C v_n ))$. Thus, by \Cref{lem:dceasy}, for all $\phi\in \Cc(\Omega)$,
\begin{multline*}
   \int_\Omega \langle q, C u_0\rangle_{\C^k} \phi = \lim_{n\to\infty}    \int_\Omega \langle a_n (C u_n +z), C v_n\rangle_{\C^k} \phi  \\ = \lim_{n\to\infty}\int_\Omega \langle (C u_n +z), a_n^\ast C v_n\rangle_{\C^k} \phi = \int_\Omega \langle (C v +z), a^\ast C u_0 \rangle_{\C^k} \phi.
\end{multline*}
Hence, 
\[\langle q, C u_0\rangle_{\C^k}=\langle (C v +z), a^\ast C u_0 \rangle_{\C^k} =\langle a(C v +z), C u_0 \rangle_{\C^k}. \]
Since $u_0\in V$ was arbitrary, we infer $q=a(C v +z)$ from \Cref{lem:localvectors}. By
\[
 \forall \phi \in V:\langle a_n (C u_n +z), C \phi\rangle_{\Lp{2}(\Omega)^k} = h(\phi),
\]
 $a_n (C u_n +z) \rightharpoonup q=a(C v +z)$ implies
\[
 \forall \phi \in V:\langle a (C v +z), C \phi\rangle_{\Lp{2}(\Omega)^k} = h(\phi).
\]
Hence, $v=u$ follows from the uniqueness of solutions.
\end{proof}
Next, we argue -- as in the classical situation -- that computing the adjoint is continuous within $\Htopo$-convergence. 
\begin{corollary}\label{cor:adjH} Consider $(a_n)_{n\in\N}$, $a$ in $M(\alpha,\beta;C)$. If $a_n\to a$ in $\tau_{\Htopo}$, then $a^\ast_n\to a^\ast$ in $\tau_{\Htopo}$.
\end{corollary}
\begin{proof} 
Choose $z=0$ in \Cref{prop:Hconplus} and use $a_n^{**}=a_n$ for $n\in\N$.
\end{proof}

\section{Nonlocal $\Htopo$-Convergence}\label{sec:non}

The integral tool in the proofs of  this paper's main statements is the following lemma about the complementary problem to~\eqref{eq:AbsElliPDE}.
\begin{lemma}\label{prop:equivprobl}
Let $\Omega\subseteq \R^d$ be open and bounded, and $0<\alpha\leq\beta$. Let $\mathcal{D}$ be as in~\eqref{eq:DefinMathCalD}, fix some closed $D_0\subseteq C\subseteq D$, and assume (com).
 
For $a \in M(\alpha,\beta; C)$ and $z\in \Lp{2}(\Omega)^k$ consider the two problems \begin{enumerate}
 \item[(i)] find $u\in V$ such that 
 \[
     \forall \phi \in V\colon \langle a (C u + z), C \phi\rangle_{\Lp{2}(\Omega)^k} =0;
 \]
 \item[(ii)] find $p\in \ran(C)^\perp\subseteq \Lp{2}(\Omega)^k$ such that 
 \[
   \forall q \in \ran(C)^\perp\subseteq \Lp{2}(\Omega)^k\colon \langle a^{-1} p, q\rangle_{\Lp{2}(\Omega)^k} = \langle z,q\rangle_{\Lp{2}(\Omega)^k}.
 \]
\end{enumerate}
Then (i) and (ii) each admit a uniquely determined solution. If $u\in V$ and $p\in \ran(C)^\perp\subseteq \Lp{2}(\Omega)^k$ are the respective solutions and $\iota_1 \colon  \ran(C)^\perp\hookrightarrow\Lp{2}(\Omega)^k$ is the canonical embedding then
\begin{equation}\label{eq:equivproblSchurCompSolves}
\iota_1(\iota_1^\ast a^{-1}\iota_1)^{-1}\iota_1^\ast z = p = a(C u + z).
\end{equation}
\end{lemma}
\begin{proof}
Unique existence of solutions is clear as well as $\iota_1(\iota_1^\ast a^{-1}\iota_1)^{-1}\iota_1^\ast z = p$.

Next, let $u\in V$ be the solution to (i). Then, $ p\coloneqq a(C u + z) \in \ran(C)^\perp\subseteq \Lp{2}(\Omega)^k$. Moreoever, for $q\in \ran(C)^\perp$ we compute
\[
   \langle a^{-1} p, q\rangle_{\Lp{2}(\Omega)^k}  = \langle a^{-1} (a(C u+z)), q\rangle_{\Lp{2}(\Omega)^k} =  \langle(C u+z), q\rangle_{\Lp{2}(\Omega)^k} = \langle z,q\rangle_{\Lp{2}(\Omega)^k},
\]as $Cu\in\ran(C)\subseteq \Lp{2}(\Omega)^k$.
\end{proof}

The following compactness statement based on results from \cite{Wa18} will be crucial for the further proofs.

 \begin{lemma}\label{lemma:CompSchurABPlusMultOp}
Consider a Hilbert space $\mathcal{H}$, a closed subspace $\mathcal{H}_0\subseteq \mathcal{H}$, and $\mathcal{H}_1\coloneqq \mathcal{H}_0^\perp$. 

For $0<\alpha\leq\beta$, the closure of
 \begin{equation*}
    \mathcal{M}(\alpha,\beta;\mathcal{H})\coloneqq \{a\in \Lb(\mathcal{H}) ; \Re a\geq \alpha, \Re a^{-1}\geq 1/\beta\}
    \end{equation*}
 in $(\mathcal{M}(\mathcal{H}_0,\mathcal{H}_1),\tau(\mathcal{H}_0,\mathcal{H}_1))$ is compact.
 \end{lemma}
\begin{proof}
First, recall the definition of $a_{jk}$ from the introduction with the canonical embeddings $\iota_0 \colon  \mathcal{H}_0\hookrightarrow\mathcal{H}$ and $\iota_1 \colon  \mathcal{H}_1\hookrightarrow\mathcal{H}$.
For $a\in \mathcal{M}(\alpha,\beta;\mathcal{H})$, simple calculations (see, e.g., \cite[Lemma~1.5.3]{Bu25} and \cite[Proposition~6.2.3]{SeTrWa22})
and the (inverted) Schur complement
  \begin{align*}
  ( a^{-1})_{11}&=(a_{11}-a_{10}a_{00}^{-1}a_{01})^{-1}  \end{align*}
show $\Re a_{00}\geq \alpha$, $\Re a_{11}-a_{10}a_{00}^{-1}a_{01}\geq \alpha$, $\Re a_{00}^{-1}\geq 1/\beta$, $\Re (a_{11}-a_{10}a_{00}^{-1}a_{01})^{-1}\geq 1/\beta$,
and $\lVert a_{10}a_{00}^{-1}\rVert,\lVert a_{00}^{-1}a_{01}\rVert\leq\beta/\alpha$.
Hence, the statement follows as a local version of \cite[Theorem~5.6]{BuSkWa24}; cf.\ also the similar original \cite[Theorem~5.5]{Wa18}.
\end{proof}

With all these tools acquired, we can provide a proof for the following main theorem of this paper, i.e., a more general version of \Cref{thm:main-proto}.
\begin{theorem}\label{thm:main-general} Let $\Omega\subseteq \R^d$ be open and bounded, and $0<\alpha\leq\beta$. Let $\mathcal{D}$ be as in~\eqref{eq:DefinMathCalD}, fix some closed $D_0\subseteq C\subseteq D$, and assume (com) and (nd). Then,
    \[(M(\alpha,\beta; C),\tau_{\Htopo})=(M(\alpha,\beta; C),\tau(\ran(C),\ran(C)^{\perp}))\text{.}\]
     In particular, for $(a_n)_{n\in\N}$ in $M(\alpha,\beta;C)$ and $a\in \Lb(\Lp{2}(\Omega)^k)$, the following conditions are equivalent:
    \begin{enumerate}
      \item[(i)] $a\in M(\alpha,\beta; C)$ and $a_n\to a$ in $\tau_{\Htopo}$;
      \item[(ii)] $a \in \mathcal{M}(\ran(C),\ran(C)^{\perp})$ and $a_n\to a$ in $\tau(\ran(C),\ran(C)^{\perp})$.
    \end{enumerate}    
\end{theorem}
\begin{proof}
First, recall the definition of $a_{jk}$ from the introduction with $\mathcal{H}\coloneqq \Lp{2}(\Omega)^k$, $\mathcal{H}_0\coloneqq\ran(C)$, and $\mathcal{H}_1\coloneqq\ran(C)^\bot$. Note that this renders the use of $\iota_0$ and $\iota_1$ for
the canonical embeddings $\iota_0 \colon  \ran(C)\hookrightarrow\Lp{2}(\Omega)^k$ and $\iota_1 \colon  \ran(C)^\perp\hookrightarrow\Lp{2}(\Omega)^k$ consistent throughout the whole paper.
 For $a\in M(\alpha,\beta; C)$, we compute,
 \begin{align*}
       a \iota_0 a_{00}^{-1} & =  \begin{pmatrix} \iota_0 & \iota_1 \end{pmatrix} \begin{pmatrix} a_{00} & a_{01} \\ a_{10} & a_{11}\end{pmatrix} \begin{pmatrix} \iota_0^* \\ \iota_1^*\end{pmatrix} \iota_0 a_{00}^{-1} \\
      &  =  \begin{pmatrix} \iota_0 & \iota_1 \end{pmatrix} \begin{pmatrix} a_{00} & a_{01} \\ a_{10} & a_{11}\end{pmatrix} \begin{pmatrix} a_{00}^{-1} \\ 0\end{pmatrix}  \\       
      & = \iota_0+ \iota_1 a_{10}a_{00}^{-1}.
      \end{align*}
By \Cref{thm:wp}  and \begin{align*}
  a\widetilde{C\kappa_0}  \operatorname{LM}_{C}(a)  & =  a\widetilde{C\kappa_0}   (\iota_0^\ast \widetilde{C\kappa_0})^{-1}(\iota_0^\ast a\iota_0)^{-1}\bigl(\bigl(\widetilde{C\kappa_0} \bigr)^\diamond\iota_0\bigr)^{-1} \\
  & =  a\iota_0 a_{00}^{-1} \bigl(\bigl(\widetilde{C\kappa_0} \bigr)^\diamond\iota_0\bigr)^{-1},
   \end{align*}
   this implies that convergence in $\tau_{\Htopo}$ is equivalent to the mappings $ a\mapsto a_{00}^{-1}$ and $a\mapsto a_{10}a_{00}^{-1}$ being continuous in their respective weak operator topologies chosen in the image space.
In other words, the identity mapping
\[(M(\alpha,\beta; C),\tau(\ran(C),\ran(C)^{\perp}))\hookrightarrow (M(\alpha,\beta; C),\tau_{\Htopo})\]
is continuous. Since both involved topologies are metric and due to \Cref{lemma:CompSchurABPlusMultOp} applied to
$( \mathcal{M}(\alpha,\beta;\Lp{2}(\Omega)^{k}), \tau(\ran(C),\ran(C)^{\perp}))$, it remains to show that $M(\alpha,\beta; C)\subseteq \mathcal{M}(\alpha,\beta;\Lp{2}(\Omega)^{k})$ is closed with respect to $\tau(\ran(C),\ran(C)^{\perp})$. 
 
 Let $(a_n)_{n\in\N}$ in $M(\alpha,\beta;C)$ and $a\in \mathcal{M}(\ran(C),\ran(C)^{\perp})$ such that
 $a_n\to a$ with respect to $\tau(\ran(C),\ran(C)^{\perp})$. By virtue of \Cref{thm:Hcom}, there exists $b\in M(\alpha,\beta; C)$
 with $a_n\to b$ in $\tau_{\Htopo}$ (for a subsequence, which we do not relabel). Our previous considerations yield $a_{n,00}^{-1}\to b_{00}^{-1}$ and $a_{n,10}a_{n,00}^{-1}\to b_{10}b_{00}^{-1}$ in the respective
 weak operator topologies. For the following computations, we refer to the inverse of $a$ written in the decomposition given by $\ran(C)\oplus \ran(C)^{\bot}$; see also \cite[Lemma 5.4]{Wa25}.
 \Cref{cor:adjH}, \Cref{prop:Hconplus} (with $h=0$), and~\eqref{eq:equivproblSchurCompSolves} from \Cref{prop:equivprobl} show 
 \[( a_n^{-1})_{11}^{-1}\to ( b^{-1})_{11}^{-1}
  \quad\text{ and }\quad a_n^{-1}\iota_1( a_n^{-1})_{11}^{-1}\iota_1^\ast\to b^{-1}\iota_1( b^{-1})_{11}^{-1}\iota_1^\ast\]
  in the respective weak operator topologies. Considering the Schur complement
  \[( a_n^{-1})_{11}^{-1}=a_{n,11}-a_{n,10}a_{n,00}^{-1}a_{n,01}\]
  and
  \begin{equation}\label{eq:AbstrCurlFlux}
   \begin{aligned}
       a&_n^{-1}\iota_1( a_n^{-1})_{11}^{-1}\\
        & =  \begin{pmatrix} \iota_0 & \iota_1 \end{pmatrix} \begin{pmatrix} a_{n,00}^{-1}+a_{n,00}^{-1}a_{n,01}(a_n^{-1})_{11}a_{n,10}a_{n,00}^{-1} & -a_{n,00}^{-1}a_{n,01} (a_n^{-1})_{11}\\ -(a_n^{-1})_{11}a_{n,10}a_{n,00}^{-1} & (a_n^{-1})_{11}\end{pmatrix} \begin{pmatrix} \iota_0^* \\ \iota_1^*\end{pmatrix} \iota_1( a_n^{-1})_{11}^{-1} \\
      &  =  \begin{pmatrix} \iota_0 & \iota_1 \end{pmatrix} \begin{pmatrix} a_{n,00}^{-1}+a_{n,00}^{-1}a_{n,01}(a_n^{-1})_{11}a_{n,10}a_{n,00}^{-1} & -a_{n,00}^{-1}a_{n,01} (a_n^{-1})_{11}\\ -(a_n^{-1})_{11}a_{n,10}a_{n,00}^{-1} & (a_n^{-1})_{11}\end{pmatrix} \begin{pmatrix} 0 \\ ( a_n^{-1})_{11}^{-1}\end{pmatrix}  \\       
      & = \iota_0 a_{n,00}^{-1}a_{n,01}+\iota_1,
      \end{aligned}
\end{equation}
  (and similarly for $b$), we deduce
  \[a_{n,00}^{-1}a_{n,01}\to b_{00}^{-1}b_{01}
  \quad\text{ and }\quad
  a_{n,11}-a_{n,10}a_{n,00}^{-1}a_{n,01}\to b_{11}-b_{10}b_{00}^{-1}b_{01}
  \]
   in the respective weak operator topologies.
  All in all, we infer $a_n\to b$ in $\tau(\ran(C),\ran(C)^{\perp})$, i.e., $a=b$.
\end{proof}

\begin{proof}[Proof of \Cref{thm:main-proto}] By \Cref{exa:nd} and \Cref{ex:comp}, conditions (com) and (nd) are satisfied for $C=D_0=\cgrad$ with $V=\cH^1(\Omega)$. Moreover, we clearly have $(M(\alpha,\beta; C),\tau_{\Htopo})=(M(\alpha,\beta; \Omega),\tau_{\Htopo})$. Hence, \Cref{thm:main-general} proves the claim.
 \end{proof}

Next, we turn to a proof of \Cref{cor:main-proto}; that is, \cite[Theorem 5.11]{Wa18} the original proof of which contained a gap.

\begin{proof}[Proof of \Cref{cor:main-proto}] By \Cref{thm:main-proto} and \Cref{thm:main-general} it remains to analyse $a\mapsto a_{00}^{-1}a_{01}$ and $a\mapsto a_{11}-a_{10}a_{00}^{-1}a_{01}=(a^{-1})_{11}^{-1}=(\iota_1^\ast a^{-1}\iota_1)^{-1}$
for $a\in M(\alpha,\beta; \Omega)$. By virtue of the Helmholtz decomposition~\eqref{eq:Helmholtz}, 
\Cref{thm:wp} yields
\[  \operatorname{LM}_{\curl}(a^{-1})= (\iota_1^\ast \widetilde{\curl\kappa_0})^{-1}(a^{-1})_{11}^{-1}\bigl(\bigl(\widetilde{\curl\kappa_0}\bigr)^\diamond\iota_1\bigr)^{-1}\text{,}\]
where $\operatorname{LM}_{\curl}(a^{-1})$ exactly is the unique solution to~\eqref{eq:ComplCurlProbl} and $\kappa_0\colon \ker(\curl)^\perp\hookrightarrow \Lp{2}(\Omega)^3$ is the canonical embedding. Furthermore, 
 \begin{align*}
  a^{-1}&\widetilde{\curl\kappa_0}  \operatorname{LM}_{\curl}(a^{-1})\\
    & =  a^{-1}\widetilde{\curl\kappa_0} (\iota_1^\ast \widetilde{\curl\kappa_0})^{-1}(a^{-1})_{11}^{-1}\bigl(\bigl(\widetilde{\curl\kappa_0}\bigr)^\diamond\iota_1\bigr)^{-1}\\
  & =  a^{-1}\iota_1 (a^{-1})_{11}^{-1}\bigl(\bigl(\widetilde{\curl\kappa_0}\bigr)^\diamond\iota_1\bigr)^{-1}.
   \end{align*}
It remains to consult~\eqref{eq:AbstrCurlFlux}.
\end{proof}

Finally, we provide a proof of \Cref{cor:curl}. For this, we recall two elementary facts from \cite{Wa25} concerning the Schur topology. 
\begin{lemma}[{{\cite[Lemma 5.4]{Wa25}}}]\label{lem:inv}
Consider a Hilbert space $\mathcal{H}$, a closed subspace $\mathcal{H}_0\subseteq \mathcal{H}$, and $\mathcal{H}_1\coloneqq \mathcal{H}_0^\perp$. 

For $0<\alpha\leq\beta$,
\[
   \big( \mathcal{M}(\alpha,\beta;\mathcal{H}), \tau(\mathcal{H}_0,\mathcal{H}_1)\big) \ni a\mapsto a^{-1} \in    \big( \mathcal{M}(1/\beta,1/\alpha;\mathcal{H}), \tau(\mathcal{H}_1,\mathcal{H}_0)\big)
\]
is a homeomorphism.
\end{lemma}
\begin{theorem}[{{\cite[Theorem 5.8]{Wa25}}}]\label{thm:fintesub}
Consider a Hilbert space $\mathcal{H}$, a closed subspace $\mathcal{H}_0\subseteq \mathcal{H}$, $\mathcal{H}_1\coloneqq \mathcal{H}_0^\perp$, and $0<\alpha\leq\beta$.

 Let $(a_n)_{n\in\N}$ and $a$ be from $ \mathcal{M}(\alpha,\beta;\mathcal{H})$ and $\mathcal{F}\subseteq \mathcal{H}_0^{\bot}$ finite-dimensional. Define $\mathcal{K}_0\coloneqq \mathcal{H}_0\oplus \mathcal{F}$ and $\mathcal{K}_1\coloneqq \mathcal{K}_0^\bot$. Then
\[
  a_n \to a \text{ in }\tau(\mathcal{H}_0,\mathcal{H}_1)\iff     a_n \to a \text{ in }\tau(\mathcal{K}_0,\mathcal{K}_1).
\]
\end{theorem}
\begin{proof}[Proof of \Cref{cor:curl}.] $a_n\to a$ in $\tau_{\Htopo}$ is equivalent to $a_n\to a$ in $\tau(g_0(\Omega),g_0(\Omega)^{\bot})$ by \Cref{thm:main-proto} (or \Cref{thm:main-general}). The Helmholtz decomposition now reads
\[
   \Lp{2}(\Omega)^3 = g_0(\Omega)\oplus \ran(\curl) \oplus \mathcal{F},
\]
where $\mathcal{F}\coloneqq\ker(\ccurl)\cap \ker(\dive)$ with $\dive\coloneqq (\cgrad)^\ast$, i.e., the usual weak divergence on its maximal $\Lp{2}$-domain. Since $\Omega$ is a weak Lipschitz domain, the Picard--Weber--Weck selection theorem (see \cite{Pi84}) yields that $\mathcal{F}\subseteq \dom(\ccurl)\cap \dom(\dive)$ is finite-dimensional. Thus, by \Cref{thm:fintesub}, we obtain that $a_n\to a$ in $\tau(g_0(\Omega),g_0(\Omega)^{\bot})$ is equivalent to 
\[
  a_n\to a \text{ in }\tau(g_0(\Omega)\oplus \mathcal{F},\ran(\curl)),
\]
which by \Cref{lem:inv} is the same as saying
\[
  a_n^{-1}\to a^{-1} \text{ in }\tau(\ran(\curl),\ran(\curl)^{\bot}).
\]
Using \Cref{thm:main-general}, we deduce that the latter is equivalent to (ii) in \Cref{cor:curl} which proves the claim.\end{proof} 

\begin{remark} Note that Free Lunch Theorems of the form of \Cref{cor:curl} are particularly interesting because both the problems considered in (i) and (ii) are divergence form problems using differential operators. This will always be the case, when the problem in (ii) of \Cref{prop:equivprobl} (modulo finite-dimensional subspace) can be written as a problem in divergence form. It is not difficult to see that this is the case when $\ran(C)^\bot$ is the range of a differential operator (modulo finite-dimensional subspace). Since $\ran(C)^\bot =\ker(C^*)$ is to be a range of a differential operator, Free Lunch Theorems with differential operators can be expected if $C$ is part of a Hilbert complex structure. This property has been used in \cite{W17_DCL,Wa18,Pauly2019} in the context of homogenisation. In \Cref{sec:NonTrivEx}, we will treat an example with differential operators of different order stemming from one Hilbert complex structure.
\end{remark}
\section{Homogenisation of a Fourth-Order Elliptic Problem}\label{sec:NonTrivEx}

In order to show the strenght of the Free Lunch Theorem and \Cref{thm:Hcom} as a generalisation of \cite[Theorem~6.5]{Ta09}, we will discuss non-periodic homogenisation of a fourth-order elliptic problem corresponding to the biharmonic equation. The appearing Hilbert complex structure(s) and differential operators have already been treated extensively in the literature, see, most notably in \cite{PZ20} and, for a homogenisation context, \cite{W17_DCL}.

We assume that $\Omega\subseteq \R^3$ is an open and bounded strong Lipschitz domain. We define the densely defined and closed second-order differential operator induced by the Hessian
\[
\cgradgrad\colon\dom(\cgradgrad)\subseteq\Lp{2}(\Omega)\to \Lp{2,\sym}(\Omega)^{3 \times 3}
\]
 to be the operator closure of
\[
\cgradgrad\restriction_{\Cc(\Omega)}\colon \begin{cases}
\hfill\Cc(\Omega)\subseteq\Lp{2}(\Omega)  &\to \Lp{2,\sym}(\Omega)^{3 \times 3}\\
\hfill \phi&\mapsto (\partial_j\partial_i \phi)_{i,j=1}^3
\end{cases}\text{.}
\]
A standard density argument with test functions shows $\dom(\cgradgrad)=\cH^2(\Omega)$, and subsequently
 $\ker(\cgradgrad)=\{0\}$ as well as closedness of $\ran(\cgradgrad)$ (compare \cite[Lemma~3.3]{PZ20}). Thus, \cite[Theorem 3.1]{TrWa14} immediately
 yields (see also \Cref{thm:wp}):
 \begin{lemma}\label{lemma:EllipticWP4thOrderGGrad}
 Let $a\in \Lb(\Lp{2,\sym}(\Omega)^{3 \times 3})$ with $\iota_0^\ast a\iota_0$ boundedly invertible, where
 \[
 \iota_0\colon \ran(\cgradgrad)\hookrightarrow \Lp{2,\sym}(\Omega)^{3 \times 3}
 \]
 is the canonical embedding.

If $f\in \sobH^{-2}(\Omega)$, then there exists a uniquely determined $u\in \cH^2(\Omega)$ such that
\begin{equation}\label{eq:WPGradGradPDE}
    \forall \phi\in \cH^2(\Omega):\langle a \cgradgrad u, \cgradgrad\phi\rangle_{\Lp{2,\sym}(\Omega)^{3 \times 3}} = f(\phi).
 \end{equation}
To be precise, we have $u= \operatorname{LM}_{\cgradgrad}(a) f$, where
\[
\Lb(\sobH^{-2}(\Omega),\cH^2(\Omega))\ni \operatorname{LM}_{\cgradgrad}(a)
\coloneqq (\iota_0^\ast \widetilde{\cgradgrad})^{-1}(\iota_0^\ast a\iota_0)^{-1}\bigl(\bigl(\widetilde{\cgradgrad} \bigr)^\diamond\iota_0\bigr)^{-1}\text{,}
\]
where
\[
\iota_0^\ast \widetilde{\cgradgrad}\colon \begin{cases}
\hfill\cH^2(\Omega)  &\to \ran(\cgradgrad)\\
\hfill \phi&\mapsto \iota_0^\ast \cgradgrad \phi
\end{cases}\text{,}
\]
and
\[
\bigl(\widetilde{\cgradgrad}\bigr)^\diamond\iota_0\colon \begin{cases}
\hfill\ran(\cgradgrad)  &\to \sobH^{-2}(\Omega)\\
\hfill q &\mapsto \Bigl(\cH^2(\Omega)\ni \phi \mapsto \langle q,\cgradgrad \phi\rangle_{\Lp{2,\sym}(\Omega)^{3 \times 3}}\Bigr)
\end{cases}\text{,}
\]
are (anti-)linear,
bounded, and have (anti-)linear and bounded inverses.
 \end{lemma}

By identifying $\Lp{2,\sym}(\Omega)^{3 \times 3}$ with $\Lp{2}(\Omega)^{6}$ (counting row-wise in the upper-triangular part) and
\[
\Lp{2,\dev}(\Omega)^{3 \times 3}\coloneqq\{\Psi\in \Lp{2}(\Omega)^{3 \times 3}:\Psi\text{ is deviatoric (i.e., trace-free)}\}
\]
with $\Lp{2}(\Omega)^{8}$ (counting row-wise and leaving out the $(3,3)$-entry),
we define
\[
\sym\curl_{\dev}\colon \dom(\sym\curl_{\dev})\subseteq \Lp{2,\dev}(\Omega)^{3 \times 3}\to \Lp{2,\sym}(\Omega)^{3 \times 3}
\]
as $D$ corresponding to $   \mathcal{D}= \sum_{j=1}^3 \partial_j A_j$ with
\begin{gather*}
A_1\coloneqq
\begin{psmallmatrix}
0&0&0&0&0&0&0&0\\
0&0&-1/2&0&0&0&0&0\\
0&1/2&0&0&0&0&0&0\\
0&0&0&0&0&-1&0&0\\
1/2&0&0&0&1&0&0&0\\
0&0&0&0&0&0&0&1
\end{psmallmatrix}\text{,}\qquad
A_2\coloneqq
\begin{psmallmatrix}
0&0&1&0&0&0&0&0\\
0&0&0&0&0&1/2&0&0\\
-1&0&0&0&-1/2&0&0&0\\
0&0&0&0&0&0&0&0\\
0&0&0&-1/2&0&0&0&0\\
0&0&0&0&0&0&-1&0
\end{psmallmatrix}\text{, and}\\
A_3\coloneqq
\begin{psmallmatrix}
0&-1&0&0&0&0&0&0\\
1/2&0&0&0&-1/2&0&0&0\\
0&0&0&0&0&0&0&-1/2\\
0&0&0&1&0&0&0&0\\
0&0&0&0&0&0&1/2&0\\
0&0&0&0&0&0&0&0
\end{psmallmatrix}\text{.}
\end{gather*}
It is straightforward to show that this indeed corresponds to the row-wise application of $\curl$ with final symmetrisation on the maximal weak $\Lp{2,\dev}(\Omega)^{3 \times 3}$-domain. Note that already $A_1$ and $A_2$ suffice to
obtain (nd). The selection theorem \cite[Lemma~3.22]{PZ20} together with the complex property \cite[Lemma~3.7]{PZ20} readily imply (com) together with the
Helmholtz decomposition
\begin{equation}\label{eq:HelmholtzGradGradSymCurlDev}
\Lp{2,\sym}(\Omega)^{3 \times 3}=\ran(\sym\curl_{\dev})\oplus\ran(\cgradgrad)\oplus\mathcal{G}\text{,}
\end{equation}
where $\mathcal{G}\coloneqq \ker((\sym\curl_{\dev})^\ast)\cap\ker((\cgradgrad)^\ast)$ is of finite dimension.  Therefore, we can apply both the entirety of the results from \Cref{sec:HConForC} and the theory of nonlocal $\Htopo$-convergence.

Here, the proof of \Cref{thm:Hcom} without assuming condition (ap) from \Cref{rem:ProofWIthFEMPropAP} is crucial since, to the best of the authors' knowledge, FEM-theory for $\sym\curl_{\dev}$ does not exist yet. The homogenisation result now reads as follows.
   \begin{theorem}\label{thm:curlhess} Let $\Omega\subseteq \R^3$ be an open and bounded strong Lipschitz domain, and $0<\alpha\leq\beta$. Let $(a_n)_{n\in\N}$ in $M(\alpha,\beta;\sym\curl_{\dev})$ and $a\in M(\alpha,\beta;\sym\curl_{\dev})$. 
    Then,
    \[
    a_n\to a\text{ in }\tau_{\Htopo}
    \]
    implies that for all $f\in \sobH^{-2}(\Omega)$ and $u_n\in \cH^2(\Omega)$, $n\in \N$, given as the unique solutions to (compare \eqref{eq:WPGradGradPDE})
    \[
       \forall \phi\in \cH^2(\Omega):\langle a_n^{-1} \cgradgrad u_n, \cgradgrad\phi\rangle_{\Lp{2,\sym}(\Omega)^{3 \times 3}} = f(\phi),
    \]
    we have
    \[
        u_n \rightharpoonup u\in \cH^2(\Omega)\text{ and } a_n^{-1}\cgradgrad u_n \rightharpoonup a^{-1}\cgradgrad u \in \Lp{2,\sym}(\Omega)^{3 \times 3},
    \]
    where $u\in \cH^2(\Omega)$ is the unique solution to
    \[
           \forall \phi\in \cH^2(\Omega):\langle a^{-1} \cgradgrad u, \cgradgrad\phi\rangle_{\Lp{2,\sym}(\Omega)^{3 \times 3}} = f(\phi).
    \]
    \end{theorem}
    \begin{proof}
    From \Cref{thm:main-general}, the Helmholtz decomposition \eqref{eq:HelmholtzGradGradSymCurlDev}, \Cref{thm:fintesub}, and \Cref{lem:inv},
    we deduce $a_n^{-1}\to a^{-1}$ in $\tau(\ran(\cgradgrad),\ran(\cgradgrad)^\perp)$. Considering $\operatorname{LM}_{\cgradgrad}(a_n^{-1})$ from
    \Cref{lemma:EllipticWP4thOrderGGrad} and the Schur convergence $(a_n^{-1})_{00}^{-1}\to (a^{-1})_{00}^{-1}$ in the corresponding weak operator topology, we obtain $u_n \rightharpoonup u\in \cH^2(\Omega)$. Finally, from the Schur convergence $(a^{-1}_n)_{10}(a^{-1}_n)_{00}^{-1}\to (a^{-1})_{10}(a^{-1})_{00}^{-1}$
     in the corresponding weak operator topology, \Cref{lemma:EllipticWP4thOrderGGrad}, and
    \begin{align*}
  a^{-1}_n\widetilde{\cgradgrad}  \operatorname{LM}_{\cgradgrad}(a^{-1}_n) 
  & =  a^{-1}_n\widetilde{\cgradgrad} (\iota_0^\ast \widetilde{\cgradgrad})^{-1}(\iota_0^\ast a^{-1}_n\iota_0)^{-1}\bigl(\bigl(\widetilde{\cgradgrad} \bigr)^\diamond\iota_0\bigr)^{-1}\\
  & =  a^{-1}_n\iota_0 (a_n^{-1})_{00}^{-1}\bigl(\bigl(\widetilde{\cgradgrad} \bigr)^\diamond\iota_0\bigr)^{-1}\text{,}
   \end{align*}
   we infer $a_n^{-1}\cgradgrad u_n \rightharpoonup a^{-1}\cgradgrad u \in \Lp{2,\sym}(\Omega)^{3 \times 3}$.
    \end{proof}

As a particular application of the Free Lunch Theorem just presented, we obtain a compactness result for the fourth order problem, which requires modifying standard techniques since the div-curl lemma as presented in \Cref{lem:dceasy} and the proof of the compactness theorem itself, \Cref{thm:Hcom}, make use of the product rule for first order differential operators. We emphasise the identification of $\Lp{2,\sym}(\Omega)^{3 \times 3}$ with $\Lp{2}(\Omega)^{6}$ and thus bounded multiplication operators of the former and $6\times 6$-matrices.
\begin{corollary}
Let $\Omega\subseteq \R^3$ be an open and bounded strong Lipschitz domain, and $0<\alpha\leq\beta$. Let $(a_n)_{n\in\N}$ be such that
\[
   b_n \in  \Lp{\infty}(\Omega)^{6 \times 6}\text{ with } \Re b_n\geq \alpha, \Re b_n^{-1}\geq 1/\beta.
\]  Then, there exists a subsequence (not relabeled) of $(b_{n})_n$ and $b \in  \Lp{\infty}(\Omega)^{6 \times 6}$ such that for all $f\in \sobH^{-2}(\Omega)$ and $u_n\in \cH^2(\Omega)$, $n\in \N$, given as the unique solutions to 
    \[
       \forall \phi\in \cH^2(\Omega):\langle b_{n} \cgradgrad u_n, \cgradgrad\phi\rangle_{\Lp{2,\sym}(\Omega)^{3 \times 3}} = f(\phi),
    \]
    we have
    \[
        u_n \rightharpoonup u\in \cH^2(\Omega)\text{ and } b_{n}\cgradgrad u_n \rightharpoonup b\cgradgrad u \in \Lp{2,\sym}(\Omega)^{3 \times 3},
    \]
    where $u\in \cH^2(\Omega)$ is the unique solution to
    \[
           \forall \phi\in \cH^2(\Omega):\langle b \cgradgrad u, \cgradgrad\phi\rangle_{\Lp{2,\sym}(\Omega)^{3 \times 3}} = f(\phi).
    \]
\end{corollary}
\begin{proof}
Note that $a_n \coloneqq b_n^{-1} \in M(1/\beta,1/\alpha;\sym\curl_{\dev})   $. Thus, we may apply \Cref{thm:Hcom} and obtain a subsequence of $(a_n)_n$ converging in $(M(1/\beta,1/\alpha;\sym\curl_{\dev}),\tau_{\Htopo})$ with corresponding limit $a\eqqcolon b^{-1}$. Hence, the claim follows from \Cref{thm:curlhess}.
\end{proof}

\section{Conclusion}\label{sec:con}

We provided an independent proof of \cite[Theorem 5.11]{Wa18}, the original proof of which contained a gap. Since \cite[Theorem 5.11]{Wa18} forms the core argument as to why nonlocal $\Htopo$-convergence is a proper generalisation of local $\Htopo$-convergence, this achievement reconfirmed all rationales that based on \cite[Theorem 5.11]{Wa18} to be true. The techniques developed for this independent results also apply to a broader class of differential operators. Thus, we extended and justified the generalisation for a bigger class of operators. In passing, we extended the definition of local $\Htopo$-convergence to a broader class of operators and provided fundamental observations like compactness and the continuity of  computing the adjoint in this context. The whole set-up also provided us with a Free Lunch Theorem, that serves us convergence of an associated class of problems given the initial class of homogenisation problems do converge in some $\Htopo$-sense. We applied this Free Lunch Theorem to show a homogenisation result for a fourth-order elliptic problem.

Future applications will now be handled much easier and with a wider impact in particular on homogenisation problems for time-dependent PDEs. We note that the new, generalised version of \cite[Theorem 5.11]{Wa18} does not explicitly require smoothness of the domain or assumptions resulting in the triviality of harmonic Dirichlet and Neumann fields, implying an immediate impact on abstract results as provided for instance in \cite{BuSkWa24}.

\section*{Data Availability}

Data sharing is not applicable to this article as no datasets were generated or analysed in this study.

\section*{Conflicts of Interest}

The authors declare no conflicts of interest.

\end{document}